%% file: main__waka0207.tex
\documentclass[11pt,a4paper,reqno]{amsart}
\usepackage[width=0.8\paperwidth,height=0.85\paperheight,twoside=false,includehead,includefoot]{geometry}
\allowdisplaybreaks[1]

\usepackage{amsmath, amssymb, amsfonts, latexsym,enumerate, arydshln, mathtools}

\usepackage{tikz}
\usetikzlibrary{decorations.pathreplacing,angles,quotes}
\usetikzlibrary{calc}
\usetikzlibrary{arrows}
\usepackage[all]{xy}

\newtheorem{thm}{Theorem}[section]
\newtheorem{lem}[thm]{Lemma}
\newtheorem{prop}[thm]{Proposition}

\theoremstyle{definition}

\newtheorem{ex}[thm]{Example}
\theoremstyle{remark}

\DeclareMathOperator{\Ker}{Ker}
\DeclareMathOperator{\Imm}{Im}

\newcommand{\Q}{\mathbb{Q}}
\newcommand{\Z}{\mathbb{Z}}

\newcommand{\I}{\mathbb{I}}

\newcommand{\calH}{\mathcal{H}}

\newcommand{\calT}{\mathcal{T}}
\newcommand{\calF}{\mathcal{F}}

\newcommand{\fH}{\mathfrak{H}}
\newcommand{\di}{\diamond}

\usepackage{shuffle}

\begin{document}

\title{On a family of relations of rooted tree maps}

\author{Hideki Murahara}
\address[Hideki Murahara]{The University of Kitakyushu, \endgraf 
4-2-1 Kitagata, Kokuraminami-ku, Kitakyushu, Fukuoka, 802-8577, Japan}
\email{hmurahara@mathformula.page}

\author{Tatsushi Tanaka}
\address[Tatsushi Tanaka]{Department of Mathematics, Faculty of Science, Kyoto Sangyo University, \endgraf 
Motoyama, Kamigamo, Kita-ku, Kyoto-city 603-8555, Japan}
\email{t.tanaka@cc.kyoto-su.ac.jp}

\author{Noriko Wakabayashi}
\address[Noriko Wakabayashi]{Center of Physics and Mathematics, Institute for Liberal Arts and Sciences, Osaka Electro-Communication University  \endgraf 
18-8, Hatsutyou, Neyagawa-city, Osaka 572-8530, Japan}
\email{wakabayashi@osakac.ac.jp}


\begin{abstract}
This paper is devoted to proving an infinite sequence of relations for rooted tree maps. 
On the way, we also give a basis for the space of rooted tree maps. 
\end{abstract}

\maketitle

\section{Introduction}
In \cite{B, D, CK}, a Hopf algebra of rooted trees $\calH$ is introduced, and in \cite{T}, it is shown that $\calH$ acts on the noncommutative polynomial ring $\fH = \Q\langle x, y \rangle$.
This action gives rise to linear maps on $\fH$, which are called rooted tree maps (RTMs). 
It is shown in \cite{T} that RTMs state a class of linear relations for multiple zeta values. 
Some other results on RTMs are found in \cite{BT1, BT2, MT, MTW, TW}. 

This paper aims to discuss the structure of the space of RTMs $\widetilde{\calH}$. 
The main results are summarized in \S\ref{sec3}. We show in \S\ref{sec4} that elements in $\widetilde{\calH}$ which come from rooted forests made only by applying the grafting operator $B_+$ or by multiplying the tree $\input{tree_1}$ of degree one form a basis of $\widetilde{\calH}$. 
Thus we find the space $\widetilde{\calH}$ is much smaller and there are a lot of relations among RTMs. 
In \S\ref{sec5}, we give a proof of a family of relations for RTMs by interpreting it as an identity on $\fH$ containing the product $\di$, 
an analogue of the usual harmonic product.

\section{Preliminaries}\label{sec2}
\subsection{Hopf algebra of rooted trees : an overview}
A rooted tree $t$ is a connected, finite, oriented, and simple graph 
such that there is precisely one distinguished vertex, called the root of $t$, which has no incoming edges. 
We here consider only non-planar rooted trees, that is, branches of $t$ are not ordered. 
The number of vertices of $t$ is called its degree and denoted by $\deg(t)$. 
Rooted trees of degree $\leq 4$ are visualized as follows. 
$$\input{tree_1}\ ,\ \ \ \  
\input{tree_2}\ , \ \ \ \ 
\input{tree_3_1}\ , \ \ \ \ \input{tree_3_2}\ ,\ \ \ \ 
\input{tree_4_1}\ , \ \ \ \ \input{tree_4_2}\ , \ \ \ \ \input{tree_4_4}=\input{tree_4_4_2}\ , \ \ \ \ \input{tree_4_3}\ .$$  
The topmost vertex of each rooted tree stands for its root. 
Let $\mathcal{T}$ be the set of (isomorphism class of) rooted trees and, 
for $n \geq 1$, put $\mathcal{T}(n)\coloneqq \{t \in \mathcal{T} | \deg(t) =n \}$. 
It is known that the number of elements of $\mathcal{T}(n)$ is given as follows. 
\begin{align*}
\begin{tabular}{c|ccccccccccc}
$n$ & $1$ & $2$ & $3$ & $4$ & $5$ & $6$ & $7$ & $8$ & $9$ & $10$ & $\cdots$ \\
\hline
$\#\mathcal{T}(n)$ & $1$ & $1$ & $2$ & $4$ & $9$ & $20$ & $48$ & $115$ & $286$ & $719$ & $\cdots$
\end{tabular}
\end{align*}
A rooted forest is a product of rooted trees. 
The product is defined by disjoint union, which is commutative. 
The unit is denoted by $\I$ and called the empty forest. 
Let $\calH$ be the free, associative, commutative, and unitary $\Q$-algebra generated by $\mathcal{T}$. 
As a vector space, we find $\displaystyle\calH=\sum \Q \cdot f$, where the sum runs over all rooted forests $f$. 
For example, we have $3\, \input{tree_4_3} + 8\, \input{tree_4_2}\ \input{tree_2} \, \in \calH$. 

Similar to the case of rooted trees, 
let $\calF$ be the set of (isomorphism class of) rooted forests. 
For $f \in \calF$, the number of vertices of $f$ is called its degree and denoted by $\deg(f)$. 
For $n \geq 0$, put 
$\calF\coloneqq \{f \in \calF | \deg(f)=n \}$. 
Then we find $\# \calF(n) = \# \calT(n+1)$ for $n \geq 0$ since one obtains an element in $\calF(n)$ by getting rid of the root of $t \in \calT(n)$, 
and conversely one obtains an element in $\calT(n+1)$ by grafting every root of connected components of $f \in \calF(n)$ to a common new root. 

Denote by $\langle \calT \rangle_{\Q}$ the $\Q$-vector space generated by $\calT$. 
The grafting operator is usually denoted by $B_+$, which is the $\Q$-linear map from $\calH$ to $\langle \calT \rangle_{\Q}$ defined by $B_+(\I)=\input{tree_1}$ and, 
for $f=t_1 \cdots t_n \in \calF$ with $t_1, \ldots, t_n \in \calT$, 
$B_+(f) = \input{tree_t}$, which represents a single rooted tree obtained by connecting each root of $t_j$'s to a new root by an edge. 
For example, we have 
$B_+\left(3\, \input{tree_4_3} + 8\, \input{tree_4_2}\ \input{tree_2}\right)= 3\, \input{tree_5_4} + 8\, \input{tree_7}\ .$ 
Note again that for any $t \in \calT$ there is a unique $f \in \calF$ such that $t=B_+(f)$.

The algebra $\calH$ is not only an algebra but a Hopf algebra. 
One way to define the coproduct $\Delta : \mathcal{H} \to \mathcal{H} \otimes \mathcal{H}$ is given algebraically by using the grafting operator $B_+$. 
The coproduct $\Delta$ is multiplicative and hence we only need to defind for elements in $\calT$. 
For $t=B_+(f) \in \calT$, we define 
$$\Delta(t) = t \otimes \I + ((\text{id} \otimes B_+)\circ \Delta)(f). $$
Then, for example, we find
\begin{align*}
\Delta(\I)& =\I \otimes \I, \\
\Delta(\input{tree_1})& 
                       = \input{tree_1} \otimes \I+ \I \otimes \input{tree_1}, \\
\Delta(\input{tree_2})
                       &= \input{tree_2} \otimes \I + \input{tree_1} \otimes \input{tree_1}+ \I \otimes \input{tree_2}, \\
\Delta(\input{tree_1}\ \input{tree_1})
                       &= \input{tree_1}\ \input{tree_1} \otimes \I +2\ \input{tree_1} \otimes \input{tree_1} + \I \otimes \input{tree_1}\ \input{tree_1}, \\
\Delta(\input{tree_3_2})
                       & = \input{tree_3_2} \otimes \I + \input{tree_1}\ \input{tree_1} \otimes \input{tree_1}+ 2\  \input{tree_1} \otimes \input{tree_2} + \I \otimes \input{tree_3_2},
                       \end{align*}
and so on. 
The last example shows that $\Delta$ is not cocommutative but it is known that $\Delta$ is coassociative, i.e., $(\text{id} \otimes \Delta) \circ \Delta = (\Delta \otimes \text{id}) \circ \Delta$. 
We do not need the counit and the antipode in this paper and hence we omit to define them. (For details, see \cite{CK}, for example.)
\subsection{Rooted tree maps}
Following \cite{T}, we here define rooted tree maps. 
Let the identity map on $\fH$ be assigned to the empty forest, i.e., $\tilde{\I} =\text{id}$. 
For any $f \in \calF \backslash \{\I\}$, 
we define the $\Q$-linear map $\tilde{f} : \fH \to \fH$ by the following conditions. 
\begin{enumerate}[(I)]
    \item If $f=\input{tree_1}$, $\tilde{f}(x) =-\tilde{f}(y)=xy$, 
    \item $\widetilde{B_+(f)}(x)=-\widetilde{B_+(f)}(y)=R_yR_{x+2y}R_y^{-1}\tilde{f}(x)$, 
    \item If $f=gh$ ($g, h \in \calF$), $\tilde{f}(v) =\tilde{g}(\tilde{h}(v))$ for $v \in \{ x, y\}$, 
    \item $\tilde{f}(wv)=M(\widetilde{\Delta(f)}(w \otimes v))$ for $w \in \fH$, $v \in \{x, y\}$,
\end{enumerate}
where $R_w$ denotes the right multiplication map by $w$, i.e. $R_w(v)=vw$ ($v, w \in \fH$), $M : \fH \otimes \fH \to \fH$ denotes the concatenation product, and $\widetilde{\Delta(f)}=\sum_{(f)}\tilde{f_1} \otimes \tilde{f_2}$ if $\Delta(f) =\sum_{(f)} f_1 \otimes f_2$, which is Sweedler's notation. 

We find that $\tilde{f}(x)$ always ends with $y$ and hence the condition (II) is well-defined. 
We also find that $\tilde{f}(v)$ in condition (III) does not depend on how to decompose $f$ into $g$ and $h$. 
We can also show that conditions (III) and (IV) hold for any $v \in \fH$.
(See \cite[Theorem 1.2]{T}.)
We call $\tilde{f}$ the RTM assigned to $f \in \calH$. 
\begin{ex}
Since $\widetilde{\input{tree_1}}\, (x)=-\, \widetilde{\input{tree_1}}\,(y)=xy$ and 
$\Delta(\input{tree_1})=\input{tree_1} \otimes \I + \I \otimes \input{tree_1}$, we have
$$\widetilde{\,\input{tree_1}\ \input{tree_1}\,}\,(x) = \widetilde{\input{tree_1}}\,(x) y + x\ \widetilde{\input{tree_1}}\,(y)=xy^2-x^2y.$$
Then we calculate
$$\widetilde{\,\input{tree_3_2}\,}(x)=\widetilde{B_+(\input{tree_1}\ \input{tree_1})}(x) = R_yR_{x+2y}R_y^{-1} \widetilde{\,\input{tree_1}\ \input{tree_1}\,}(x) =-x^3y-2x^2y^2+xyxy+2xy^3.$$
\end{ex}
\subsection{The algebra $\fH^1_{\di}$}
Following \cite{HMO}, we define the product $\di : \fH \times \fH \to \fH$ by $\Q$-bilinearlity and
\begin{align*}
w \di 1 & = 1 \di w = w, \\
vx \di wx & = (v \di wx)x -(vy \di w)x, \\
vx \di wy & = (v \di wy)x +(vx \di w)y, \\
vy \di wx & = (v \di wx)y +(vy \di w)x, \\
vy \di wy & = (v \di wy)y -(vx \di w)y
\end{align*}
for $v, w \in \fH$. 
Then we find that $\fH$ is an associative and commutative $\Q$-algebra with respect to the product $\di$, 
and $\fH^1\coloneqq \Q+\fH y$ is its subalgebra which is denoted by $\fH^1_\di$. 
The product $\di$ is closely related to the usual harmonic product on $\fH$. 
Put $z\coloneqq x+y$. A key property is 
\begin{align}\label{di1}
vz \di w = v \di wz = (v\di w) z 
\end{align}
for $v,w\in\mathfrak{H}$, and thus we also find 
\begin{align}\label{di2}
(vz^{k-1} y \diamond w z^{l-1}y)
 =(v \diamond w z^{l-1}y)z^{k-1}y+(vz^{k-1} y \diamond w)z^{l-1}y-(v\diamond w)z^{k+l-1}y
\end{align}
for $k, l \geq 1$ and $v,w\in\mathfrak{H}$. 

For $f \in \calF$, we set a polynomial $F_f \in \fH^1$ recursively by 
\begin{align*}
F_\I & =1,\ F_{\input{tree_1}} =y, \\
F_t & = R_yR_{x+2y}R_y^{-1}(F_f)\ \ \text{if}\ \ t=B_+(f)\ \ \text{with}\ \ f \in \calF \backslash \{ \I \}, \\
F_f & = F_g \di F_h \ \ \text{if}\  \ f=gh.
\end{align*}
The subscript of $F$ is extended linearly. 
Then we have the following theorem. 
\begin{thm}(\cite[Theorem 1]{MT})\label{MT}
For any $f \in \calH$ and any $w \in \fH$, we have
$$\tilde{f}(xw) = x(F_f \di w). $$
\end{thm}
\section{Main theorems}\label{sec3}
Let $\rho : \calH \to \mathrm{End}(\fH)$ be $\rho(f) = \tilde{f}$. 
The map $\rho$ is an algebra homomorphism and  
$\widetilde{\calH}\coloneqq \Imm \rho$ is the space of rooted tree maps. 
Let $\mathcal{U}_0\coloneqq\{\I\}$ and 
$\mathcal{U}_d\coloneqq B_+(\mathcal{U}_{d-1}) \cup \input{tree_1} \ \mathcal{U}_{d-1}$ ($d \geq 1$), where
we set 
$$B_+(X)\coloneqq\{B_+(f)\mid f \in X \} \ \ \text{and} \ \ \input{tree_1}\, X\coloneqq\{ \input{tree_1}\, f\mid  f \in X\}$$
for $X \subset \calH$.

\begin{thm}\label{main1}
The set
$\rho\left(\bigsqcup_{d \geq0} \mathcal{U}_{d}\right)$ forms a basis of $\widetilde{\calH}$.
\end{thm}

For $ m, n \geq 1$, let
$$f_{m, n}\coloneqq 
  \hspace{15pt}\hspace{-15pt}\input{tree_m}
  \ \input{tree_n} 
  - \sum_{{\scriptsize \substack{0 \leq i \leq m-1 \\ 0 \leq j \leq n-1}}} 
  \input{tree_nm_6}\hspace{-40pt}\hspace{40pt}
  + \sum_{{\scriptsize \substack{0 \leq i \leq m-1 \\ 0 \leq j \leq n-1 \\ (i, j) \neq (0, 0)}}} 
 \input{tree_nm_7}\hspace{-40pt}\hspace{40pt} \in \calH.  
 $$
For example,
\begin{align*}
f_{2, 2} & = \input{tree_2}\ \input{tree_2} - 2 \, \input{tree_4_4}- \input{tree_1}\ \input{tree_3_2}+ \input{tree_4_2}+ \input{tree_4_3}\, ,\\
f_{2, 3} & = \input{tree_2} \ \input{tree_3_1} -\input{tree_5_5} 
-\input{tree_1}\ \input{tree_4_4} - \input{tree_5_6}
- \input{tree_5_3} + \input{tree_5_8} +\input{tree_5_4} + \input{tree_5_2}\, .
\end{align*}
\begin{thm}\label{main2}
For any $m, n \geq 1$, we have $\rho(f_{m, n})=0$. 
\end{thm}

Let $\sigma : \calH \to \fH^1_{\diamond}$ be $\sigma(f)=F_{f}$.
The map $\sigma$ is a surjective algebra homomorphism.  
Theorem \ref{main2} turns out to be equivalent to the following theorem. 
\begin{thm}\label{main3}
For any $m, n \geq 1$, we have $\sigma(f_{m, n})=0$. 
\end{thm}
\section{Proof of Theorem \ref{main1}}\label{sec4}
\begin{lem}\label{lem1}
For any $f \in \calH$, $\tilde{f}=0$ if and only if $\tilde{f}(x)=0$.
\end{lem}
\begin{proof}
It is obvious that $\tilde{f}(x)=0$ if $\tilde{f}=0$. 
Suppose $\tilde{f}(x)=0$. By definition, we see that $\tilde{f}(x+y)=0$, and hence $\tilde{f}(y)=0$. By \cite{BT2}, for any $w \in x\fH y$
there exists $g \in \calH$ such that $\tilde{g}(x)=w$. 
Hence
$$\tilde{f}(w)=\tilde{f}(\tilde{g}(x))=\tilde{g}(\tilde{f}(x))=\tilde{g}(0)=0. $$
Combining this with
\begin{align*}
\tilde{f}(yw) &= (x+y)\tilde{f}(w)-\tilde{f}(xw), \\
\tilde{f}(wx) &= \tilde{f}(w) (x+y)-\tilde{f}(wy) 
\end{align*}
($w \in \fH$),  which come from \cite[Theorem 1.2]{T}, 
we conclude $\tilde{f}(w) =0$ for all $w \in \fH$ inductively. It is easy to see that $\tilde{f}(\Q)=\{0\}$ if $f\in\calH$ has no terms of degree $0$. 
\end{proof}

\begin{lem}\label{lem2}
$\Ker \rho = \Ker \sigma$. 
\end{lem}
\begin{proof}
By \cite{MT}, $\tilde{f}(x)=xF_{f}$ for $f \in \calH$, 
and hence $\tilde{f}(x)=0$ if and only if $F_f=0$. 
Therefore we have the assertion by Lemma \ref{lem1}.
\end{proof}
\begin{proof}[Proof of Theorem \ref{main1}]
By Lemma \ref{lem2}, we have
$$\widetilde{\calH} = \Imm \rho \simeq \calH/\Ker \rho =  \calH/\Ker \sigma \simeq \Imm \sigma = \fH^1_{\diamond}.$$
Define the vector $f_d \in \mathcal{U}_{d}^{2^{d-1}}$ by $f_1 = (\,\input{tree_1}\,)$ and 
$$f_d = 
\left(
\begin{array}{c}
    B_+(f_{d-1}) \\ \hdashline 
    \input{tree_1}\,f_{d-1}
\end{array}
\right),
$$
and set 
$$F_d\coloneqq (F_{f_{d,j}})_{1\leq j\leq 2^{d-1}} \in \left( \fH^1_{\diamond} \right)^{2^{d-1}} \text{for}\  f_d = (f_{d,j})_{1\leq j\leq 2^{d-1}} \in \mathcal{U}_{d}^{2^{d-1}}. $$
Then,  by \cite[Lemma 3.6]{BT2} and Theorem \ref{MT}, we have 
$$F_d \equiv B_d w_{d} \mod 2, $$
where we denote for $d \geq 1$ by $w_d \in \left( \fH^1_{\diamond} \right)_d^{2^{d-1}}$ 
the vector of all $2^{d-1}$ words in $\fH^1_{\diamond}$ of degree $d$, ordered in lexicographical order ($x < y$) 
from the top to the bottom, and by $B_d \in M_{2^{d-1}}(\Z)$ same matrices defined in \cite{BT2}. 
Since the matrices $B_d$ modulo $2$ are invertible  because of \cite[Lemma 3.5]{BT2}, we see that
$\sigma\left(\bigsqcup_{d \geq0} \mathcal{U}_{d}\right)$ forms a basis of $\fH^1_{\diamond}$. 
Hence we find $\bigsqcup_{d \geq0} \mathcal{U}_{d}$ forms a basis of 
$\calH/\Ker \rho =  \calH/\Ker \sigma$, and the theorem. 
\end{proof}
\section{Proof of Theorem \ref{main3}}\label{sec5}
Theorem \ref{main2} gives a family of relations for RTMs. 
Because of Lemma \ref{lem2}, it is sufficient to prove Theorem \ref{main3} instead of Theorem \ref{main2}. To do this,  we prove the following lemma and proposition. 
For $k \geq 1, \mathbf{k}=(k_1, \ldots, k_r) \in \Z_{\geq 1}^r$ and formal parameters $X$ and $Y$, put
\begin{align*}
 f(k) & =X^{k}+Y^{k}-\sum_{n=1}^{k-1}X^{n}Y^{k-n},\\
 z_{k} & =f(k)z^{k-1}y,\\
 z_{\mathbf{k}} 
  & =z_{k_{1},\dots,k_{r}}=z_{k_{1}}\cdots z_{k_{r}},\quad z_{\emptyset}=1.
\end{align*}

\begin{lem}\label{lem3}
In $\fH^1_{\di}[[X, Y]]$, We have
\begin{align*}
 \frac{1}{1-RX}(1)\diamond\frac{1}{1-RY}(1)
 =\sum_{\mathbf{k}}z_{\mathbf{k}}, 
\end{align*}
where the sum runs over all indices with positive integer components and $R\coloneqq R_yR_{z+y}R^{-1}_y$.
\end{lem}
\begin{proof}
Let $w_{k}=z^{k-1}y$, $w_{\mathbf{k}} =w_{k_{1},\dots,k_{r}}=w_{k_{1}}\cdots w_{k_{r}}$, and $w_{\emptyset}=1$.
Note that
\begin{align*}
\frac{1}{1-RX}(1)
=1+\sum_{m \geq 1} R^m(1)X^m
=1+\frac{1}{1-(z+y)X}yX
=\sum_{\mathbf{k}}w_{\mathbf{k}} X^{\text{wt}(\mathbf{k})}
=\sum_{n \geq 0}\left(\sum_{k \geq 1} z^{k-1}y X^k \right)^n.
\end{align*}
%
Then, using \eqref{di2}, we calculate 
\begin{align*}
L  \coloneqq &  \frac{1}{1-RX}(1)\diamond\frac{1}{1-RY}(1)\\
  = & \sum_{n \geq 0}\left(\sum_{k \geq 1} z^{k-1}y X^k \right)^{n} 
              \diamond 
       \sum_{m \geq 0}\left(\sum_{j \geq 1} z^{j-1}y Y^j \right)^{m}\\
  = & 1+ \sum_{n \geq 1}\left(\sum_{k \geq 1} z^{k-1}y X^k \right)^n 
      + \sum_{m \geq 1}\left(\sum_{j \geq 1} z^{j-1}y Y^j \right)^m\\
      & +\left\{ \sum_{n \geq 0}\left(\sum_{k \geq 1} z^{k-1}y X^k \right)^{n} 
               \diamond 
               \sum_{m \geq 1}\left(\sum_{j \geq 1} z^{j-1}y Y^j \right)^m \right\}
        \sum_{k\geq 1} z^{k-1}yX^k\\
      & +\left\{ \sum_{n \geq 1}\left(\sum_{k \geq 1} z^{k-1}y X^k \right)^{n} 
               \diamond 
               \sum_{m \geq 0}\left(\sum_{j \geq 1} z^{j-1}y Y^j \right)^{m} \right\} 
        \sum_{j\geq 1} z^{j-1}yY^j\\
     & -\left\{ \sum_{n \geq 0}\left(\sum_{k \geq 1} z^{k-1}y X^k \right)^{n} 
              \diamond 
              \sum_{m \geq 0}\left(\sum_{j \geq 1} z^{j-1}y Y^j \right)^{m} \right\} 
         \sum_{k, j \geq 1} z^{k+j-1}yX^kY^j\\
= &  1+ \sum_{n \geq 1}\left(\sum_{k \geq 1} z^{k-1}y X^k \right)^n 
      + \sum_{m \geq 1}\left(\sum_{j \geq 1} z^{j-1}y Y^j \right)^m\\
      & + \left\{L- \sum_{n \geq 0}\left(\sum_{k \geq 1} z^{k-1}y X^k \right)^n\right\}\sum_{k \geq 1} z^{k-1}y X^k 
        + \left\{L- \sum_{m \geq 0}\left(\sum_{j \geq 1} z^{j-1}y Y^j \right)^m\right\}\sum_{j \geq 1} z^{j-1}y Y^j\\
      & - L \sum_{N \geq 1}\sum_{k=1}^{N-1} z^{N-1}yX^kY^{N-k}\\
= & 1     
      + L\sum_{k \geq 1} z^{k-1}y X^k 
        + L\sum_{j \geq 1} z^{j-1}y Y^j
         - L \sum_{N \geq 1} z^{N-1}y \sum_{k=1}^{N-1}X^kY^{N-k}\\
   = & 1     
      + L \sum_{k \geq 1}f(k) z^{k-1}y.     
\end{align*}
Therefore we conclude 
$$L=\left(1- \sum_{k \geq 1}f(k) z^{k-1}y\right)^{-1} =\sum_{j \geq 0} \left(\sum_{k \geq 1}f(k) z^{k-1}y\right)^j=\sum_{\mathbf{k}}z_{\mathbf{k}}.
$$
\end{proof}
\begin{prop} \label{theorem:a}
In $\fH^1_\di[[X, Y]]$, the identity
\begin{align*}
 & (1-RX)(1-RY)\left(\frac{1}{1-RX}(1)\diamond\frac{1}{1-RY}(1)\right)\\
 & =1+XY\Biggl(y\diamond R\left(\frac{1}{1-RX}(1)\diamond\frac{1}{1-RY}(1)\right)-R\left(y\diamond\frac{1}{1-RX}(1)\diamond\frac{1}{1-RY}(1)\right)\Biggr)
\end{align*}
holds. 
\end{prop}

\begin{proof}
If $\mathbf{k}=\emptyset$, 
\begin{align}\label{eq3.5}
 \left( y\diamond R(z_{\mathbf{k}})-R\left(y\diamond z_{\mathbf{k}}\right)\right)XY-(1-RX)(1-RY)(z_{\mathbf{k}})
  = -3zyXY-1+f(1)y. 
\end{align}
Now suppose $\mathbf{k}=(k_1, \ldots, k_r) \neq \emptyset$. 
Since
\begin{align}
 y\diamond wy=wy^{2}-wzy+\left(y\diamond w\right)y \label{eq3.2},
\end{align}
we have
$$y\diamond R(z_{\mathbf{k}}) 
  =f(k_{r})\left(z_{k_{1},\dots,k_{r-1}}z^{k_{r}-1}(z+y)y^{2}-z_{k_{1},\dots,k_{r-1}}z^{k_{r}-1}(z+y)zy+\left(y\diamond z_{k_{1},\dots,k_{r-1}}z^{k_{r}-1}(z+y)\right)y\right) $$
and 
$$ R\left(y\diamond z_{\mathbf{k}}\right) 
  =f(k_{r})\left(z_{k_{1},\dots,k_{r-1}}z^{k_{r}-1}y(z+y)y-z_{k_{1},\dots,k_{r-1}}z^{k_{r}}(z+y)y+\left(y\diamond z_{k_{1},\dots,k_{r-1}}\right)z^{k_{r}-1}(z+y)y\right).
$$
Because of \eqref{di1} and \eqref{eq3.2}, we have 
\begin{align*}
& y\diamond z_{k_{1},\dots,k_{r-1}}z^{k_{r}-1}(z+y) - \left( y\diamond z_{k_{1},\dots,k_{r-1}}\right)z^{k_{r}-1}(z+y) \\
&= y \diamond z_{k_{1},\dots,k_{r-1}}z^{k_{r}-1}y - (y \diamond z_{k_{1},\dots,k_{r-1}}z^{k_{r}-1}) y\\
&= z_{k_{1},\dots,k_{r-1}}z^{k_{r}-1}y^2-z_{k_{1},\dots,k_{r-1}}z^{k_{r}-1}zy,
\end{align*}
and hence we have
$$y \diamond R(z_{\mathbf{k}}) - R\left(y\diamond z_{\mathbf{k}}\right)
 =f(k_{r})z_{k_{1},\dots,k_{r-1}}z^{k_{r}-1}\left((z+y)y-2yz\right)y.
$$
Taking
\begin{align*} 
  (1-RX)(1-RY)(z_{\mathbf{k}}) & =(1-RX)(1-RY)\left(z_{k_{1},\dots,k_{r-1}}f(k_{r})z^{k_{r}-1}y\right)\\
  & =f(k_{r})z_{k_{1},\dots,k_{r-1}}z^{k_{r}-1}\left(1-f(1)(z+y)+(z+y)^{2}XY\right)y 
\end{align*}
into account, 
we have
\begin{align}\label{eq3.4}
&\left( y\diamond R(z_{\mathbf{k}})-R\left(y\diamond z_{\mathbf{k}}\right)\right)XY-(1-RX)(1-RY)(z_{\mathbf{k}})\nonumber\\
&=f(k_{r})z_{k_{1},\dots,k_{r-1}}z^{k_{r}-1}\left\{ -1+f(1)(z+y)-(z+3y)zXY \right\}y.
\end{align}
By the way, according to  Lemma \ref{lem3}, the identity we should show is 
\begin{align}\label{eq3.4.1}
1 + \sum_{\mathbf{k}}\left\{ (y\diamond R(z_{\mathbf{k}})-R\left(y\diamond z_{\mathbf{k}}\right))XY-(1-RX)(1-RY)(z_{\mathbf{k}})\right\}
\end{align}
is zero. 
By \eqref{eq3.5} and \eqref{eq3.4}, we have
\begin{align}\label{eq3.7}
\eqref{eq3.4.1}
 =& -3zyXY+f(1)y+\sum_{r \geq 1} \sum_{k_1, \ldots k_r \geq 1}f(k_{r})z_{k_{1},\dots,k_{r-1}}z^{k_{r}-1}\left\{ -1+f(1)(z+y)-(z+3y)zXY \right\}y \nonumber\\
  =& -3zyXY
 -\sum_{k \geq 2}f(k)z^{k-1}y + \sum_{k \geq 1}f(k)z^{k-1}\left\{ f(1)(z+y)-(z+3y)zXY \right\}y\nonumber\\
 &+\sum_{r \geq 2} \sum_{k_1, \ldots k_r \geq 1}f(k_{r})z_{k_{1},\dots,k_{r-1}}z^{k_{r}-1}\left\{ -1+f(1)(z+y)-(z+3y)zXY \right\}y\nonumber\\
  =& -3zyXY -\sum_{k \geq 2}f(k)z^{k-1}y + \sum_{k \geq 1}f(k)\left\{ z^{k}(X+Y)-z^{k+1}XY \right\}y\nonumber\\
  & + \sum_{k \geq 1}f(k)z^{k-1}\left\{ y(X+Y)-3yzXY \right\}y\nonumber\\
 &+\sum_{r \geq 2} \sum_{k_1, \ldots k_r \geq 1}f(k_{r})z_{k_{1},\dots,k_{r-1}}z^{k_{r}-1}\left\{ -1+f(1)(z+y)-(z+3y)zXY \right\}y\nonumber\\
   = & \sum_{k \geq 1}f(k)z^{k-1}\left\{ y(X+Y)-3yzXY \right\}y\nonumber\\
 &+\sum_{r \geq 2} \sum_{k_1, \ldots k_r \geq 1}f(k_{r})z_{k_{1},\dots,k_{r-1}}z^{k_{r}-1}\left\{ -1+f(1)(z+y)-(z+3y)zXY \right\}y.
\end{align}
The last equality follows from 
\begin{align*}
 & -3zyXY-\sum_{k \geq 2}f(k)z^{k-1}y+\sum_{k \geq 1}f(k)\left\{ z^{k}(X+Y)-z^{k+1}XY\right\}y\\
 & =-3zyXY-\sum_{k \geq 1}\left\{ f(k+1)z^{k}y-f(k) z^{k}(X+Y)y \right\}
  -\sum_{k \geq 2}f(k-1)z^{k}XYy\\
 & = \left\{ -3XY-f(2)+(X+Y)^2 \right\}zy-\sum_{k \geq 2}\left\{ f(k+1)-f(k)(X+Y)+f(k-1)XY\right\}z^{k}y\\
 & =0,
\end{align*}
which is due to 
\begin{align}\label{Seq1}
f(2)-\left(X+Y\right)^{2}+3XY =0
\end{align}
and
\begin{align}\label{Seq2}
f(k+1)-f(k)\left(X+Y\right)+f(k-1)XY  =0 
\end{align}
for $k \geq 2 $. Then, 
\begin{align}\label{eq3.8}
\eqref{eq3.7} 
=& -\sum_{r \geq 2} \sum_{k_1, \ldots k_r \geq 1}f(k_{r})z_{k_{1},\dots,k_{r-1}}z^{k_{r}-1}y\nonumber\\
& +\sum_{r \geq 1} \sum_{k_1, \ldots k_r \geq 1}f(k_{r})z_{k_{1},\dots,k_{r-1}}z^{k_{r}-1}\left\{ y(X+Y)-3yzXY \right\}y\nonumber\\
& +\sum_{r \geq 2} \sum_{k_1, \ldots k_r \geq 1}f(k_{r})z_{k_{1},\dots,k_{r-1}}z^{k_{r}-1}\left\{ z(X+Y) - z^2XY \right\}y\nonumber\\
=
& \sum_{r \geq 1} \sum_{k_1, \ldots k_r \geq 1}f(k_{r})z_{k_{1},\dots,k_{r-1}}z^{k_{r}-1}\left\{ y(X+Y)-3yzXY \right\}y\nonumber\\
& +\sum_{r \geq 2} \sum_{k_1, \ldots k_{r-1} \geq 1}z_{k_{1},\dots,k_{r-1}}
\left\{ -\sum_{k_r \geq 0}f(k_{r}+1)z^{k_{r}}+\sum_{k_r \geq 1 }f(k_{r})z^{k_{r}}(X+Y) -\sum_{k_r \geq 2 } f(k_r-1)z^{k_r}XY\right\}y.
\end{align}
Using \eqref{Seq2},  
\begin{align*}
\eqref{eq3.8}
= & \sum_{r \geq 1} \sum_{k_1, \ldots k_r \geq 1}z_{k_{1},\dots,k_{r}}\left\{(X+Y)-3zXY \right\}y\\
& +\sum_{r \geq 2} \sum_{k_1, \ldots k_{r-1} \geq 1}z_{k_{1},\dots,k_{r-1}}
\left\{ -f(1)-f(2)z+f(1)z(X+Y) \right\}y\\
= &\sum_{r \geq 1} \sum_{k_1, \ldots k_r \geq 1}z_{k_{1},\dots,k_{r}}z\left\{-3XY -f(2)+(X+Y)^2 \right\}y\\
=&0.
\end{align*}
The last equality is due to \eqref{Seq1}.
Hence we obtain the identity.
\end{proof}
\begin{proof}[Proof of Theorem \ref{main3}]
Compare the coefficients of $X^mY^n$ ($m, n \geq 1$) of the identity 
obtained by applying $\frac{1}{(1-RX)(1-RY)}$ to both sides of the identity stated in Proposition \ref{theorem:a}, and we have 
\begin{align*}
&R^{m-1}(y) \diamond R^{n-1}(y) \\
& = \sum_{{\scriptsize \substack{0\leq i \leq m-1,\\ 0\leq j \leq n-1}}} R^{m-i+n-j-2}\left(y \diamond R\left(R^{i-1}(y)\diamond R^{j-1}(y)\right) \right) -\sum_{{\scriptsize \substack{0\leq i \leq m-1,\\ 0\leq j \leq n-1, \\(i, j) \neq (0, 0)}}} R^{m-i+n-j-1}\left(y \diamond R^{i-1}(y) \diamond R^{j-1}(y)\right).
\end{align*}
This equation is nothing but the formula $\sigma(f_{m, n})=0$.
\end{proof}

\section*{Acknowledgement}
This work is supported by JSPS KAKENHI Grant Numbers JP22K13897 and JP23K03059.


\end{document}

%% file: tree_1.tex
\begin{tikzpicture}[scale=0.25,baseline={([yshift=-.5ex]current bounding box.center)}]
\def\cz{5}
\def\wi{0.5}

\newcommand{\ci}[1]{	
	\fill[black] (#1) circle (\cz pt);
	\draw (#1) circle (\cz pt);
}

\coordinate (R) at (0,0);

\ci{R}

\end{tikzpicture}

%% file: tree_2.tex
\begin{tikzpicture}[scale=0.25,baseline={([yshift=-.5ex]current bounding box.center)}]
\def\cz{5}
\def\wi{0.5}

\newcommand{\ci}[1]{	
	\fill[black] (#1) circle (\cz pt);
	\draw (#1) circle (\cz pt);
}

\coordinate (R) at (0,0);

\coordinate (r1) at (0,-1);

\draw (R) to (r1);

\ci{R}
\ci{r1}
\end{tikzpicture}

%% file: tree_3_1.tex
\begin{tikzpicture}[scale=0.25,baseline={([yshift=-.5ex]current bounding box.center)}]
\def\cz{5}
\def\wi{0.5}

\newcommand{\ci}[1]{	
	\fill[black] (#1) circle (\cz pt);
	\draw (#1) circle (\cz pt);
}

\coordinate (R) at (0,0);

\coordinate (r1) at (0,-1);
\coordinate (r2) at (0,-2);

\draw (R) to (r1);
\draw (r1) to (r2);

\ci{R}
\ci{r1}
\ci{r2}
\end{tikzpicture}

%% file: tree_3_2.tex
\begin{tikzpicture}[scale=0.25,baseline={([yshift=-.5ex]current bounding box.center)}]
\def\cz{5}
\def\wi{0.5}

\newcommand{\ci}[1]{	
	\fill[black] (#1) circle (\cz pt);
	\draw (#1) circle (\cz pt);
}

\coordinate (R) at (0,0);

\coordinate (r1) at (\wi,-1);
\coordinate (l1) at (-\wi,-1);

\draw (R) to (r1);
\draw (R) to (l1);

\ci{R}
\ci{r1}
\ci{l1}
\end{tikzpicture}

%% file: tree_4_1.tex
\begin{tikzpicture}[scale=0.25,baseline={([yshift=-.5ex]current bounding box.center)}]
\def\cz{5}
\def\wi{0.5}

\newcommand{\ci}[1]{	
	\fill[black] (#1) circle (\cz pt);
	\draw (#1) circle (\cz pt);
}

\coordinate (R) at (0,0);

\coordinate (r1) at (0,-1);
\coordinate (r2) at (0,-2);
\coordinate (r3) at (0,-3);

\draw (R) to (r1);
\draw (r1) to (r2);
\draw (r2) to (r3);

\ci{R}
\ci{r1}
\ci{r2}
\ci{r3}
\end{tikzpicture}

%% file: tree_4_2.tex
\begin{tikzpicture}[scale=0.25,baseline={([yshift=-.5ex]current bounding box.center)}]
\def\cz{5}
\def\wi{0.5}

\newcommand{\ci}[1]{	
	\fill[black] (#1) circle (\cz pt);
	\draw (#1) circle (\cz pt);
}

\coordinate (R) at (0,0);

\coordinate (r1) at (0,-1);

\coordinate (r1r) at (-\wi,-2);
\coordinate (r1l) at (\wi,-2);

\draw (R) to (r1);
\draw (r1) to (r1r);
\draw (r1) to (r1l);

\ci{R}
\ci{r1}
\ci{r1r}
\ci{r1l}
\end{tikzpicture}

%% file: tree_4_4.tex
\begin{tikzpicture}[scale=0.25,baseline={([yshift=-.5ex]current bounding box.center)}]
\def\cz{5}
\def\wi{0.5}

\newcommand{\ci}[1]{	
	\fill[black] (#1) circle (\cz pt);
	\draw (#1) circle (\cz pt);
}

\coordinate (R) at (0,0);

\coordinate (l1) at (-\wi,-1);
\coordinate (r1) at (\wi,-1);

\coordinate (r2) at (\wi,-2);

\draw (R) to (r1);
\draw (R) to (l1);
\draw (r1) to (r2);

\ci{R}
\ci{r1}
\ci{l1}
\ci{r2}
\end{tikzpicture}

%% file: tree_4_4_2.tex
\begin{tikzpicture}[scale=0.25,baseline={([yshift=-.5ex]current bounding box.center)}]
\def\cz{5}
\def\wi{0.5}

\newcommand{\ci}[1]{	
	\fill[black] (#1) circle (\cz pt);
	\draw (#1) circle (\cz pt);
}

\coordinate (R) at (0,0);

\coordinate (l1) at (-\wi,-1);
\coordinate (r1) at (\wi,-1);

\coordinate (l2) at (-\wi,-2);

\draw (R) to (r1);
\draw (R) to (l1);
\draw (l1) to (l2);

\ci{R}
\ci{r1}
\ci{l1}
\ci{l2}
\end{tikzpicture}

%% file: tree_4_3.tex
\begin{tikzpicture}[scale=0.25,baseline={([yshift=-.5ex]current bounding box.center)}]
\def\cz{5}
\def\wi{0.5}

\newcommand{\ci}[1]{	
	\fill[black] (#1) circle (\cz pt);
	\draw (#1) circle (\cz pt);
}

\coordinate (R) at (0,0);

\coordinate (r1) at (0,-1);
\coordinate (r1r) at (-\wi-\wi,-1);
\coordinate (r1l) at (\wi+\wi,-1);

\draw (R) to (r1);
\draw (R) to (r1r);
\draw (R) to (r1l);

\ci{R}
\ci{r1}
\ci{r1r}
\ci{r1l}
\end{tikzpicture}

%% file: tree_t.tex
\begin{tikzpicture}[scale=0.25,baseline={([yshift=-.5ex]current bounding box.center)}]
\def\cz{5}
\def\wi{1.5}

\newcommand{\ci}[1]{	
	\fill[black] (#1) circle (\cz pt);
	\draw (#1) circle (\cz pt);
}

\coordinate (R) at (0,0);

\node (r1) at (0,-1) [below] {$\cdots$} ;
\node (r1r) at (-\wi-\wi,-1) [below] {$t_1$} ;
\node (r1l) at (\wi+\wi,-1) [below] {$t_n$} ;

\draw (R) to (r1r);
\draw (R) to (r1l);

\ci{R}
\end{tikzpicture}

%% file: tree_5_4.tex
\begin{tikzpicture}[scale=0.25,baseline={([yshift=-.5ex]current bounding box.center)}]
\def\cz{5}
\def\wi{0.5}

\newcommand{\ci}[1]{	
	\fill[black] (#1) circle (\cz pt);
	\draw (#1) circle (\cz pt);
}

\coordinate (R) at (0,0);

\coordinate (r1) at (0,-1);
\coordinate (r2) at (\wi+\wi,-2);
\coordinate (c2) at (0,-2);
\coordinate (l2) at (-\wi-\wi,-2);

\draw (R) to (r1);
\draw (r1) to (r2);
\draw (r1) to (c2);
\draw (r1) to (l2);

\ci{R}
\ci{r1}
\ci{r2}
\ci{c2}
\ci{l2}
\end{tikzpicture}

%% file: tree_7.tex
\begin{tikzpicture}[scale=0.25,baseline={([yshift=-.5ex]current bounding box.center)}]
\def\cz{5}
\def\wi{0.5}

\newcommand{\ci}[1]{	
	\fill[black] (#1) circle (\cz pt);
	\draw (#1) circle (\cz pt);
}

\coordinate (R) at (0,0);

\coordinate (r0l) at (-1,-1);
\coordinate (r0r) at (1,-1);
\coordinate (r0l2) at (-1,-2);
\coordinate (r1) at (1,-2);

\coordinate (r1r) at (-0.5,-3);
\coordinate (r1l) at (-1.5,-3);

\draw (R) to (r0l);
\draw (R) to (r0r);
\draw (r0r) to (r1);
\draw (r0l) to (r0l2);
\draw (r0l2) to (r1r);
\draw (r0l2) to (r1l);

\ci{R}
\ci{r0l}
\ci{r0r}
\ci{r0l2}
\ci{r1}
\ci{r1r}
\ci{r1l}
\end{tikzpicture}

%% file: tree_m.tex
\begin{tikzpicture}[scale=0.25,baseline={([yshift=-.5ex]current bounding box.center)}]
\def\cz{5}
\def\wi{0.5}

\newcommand{\ci}[1]{	
	\fill[black] (#1) circle (\cz pt);
	\draw (#1) circle (\cz pt);
}

\coordinate (R) at (0,0);

\coordinate (r1) at (0,-1);
\coordinate (r2) at (0,-2);
\coordinate (r3) at (0,-4);
\coordinate (r4) at (0,-5);

\draw (R) to (r1);
\draw (r1) to (r2);
\draw[dotted] (r2) to (r3);
\draw (r3) to (r4);

\ci{R}
\ci{r1}
\ci{r2}
\ci{r3}
\ci{r4}

\draw[decoration={brace,raise=5pt},decorate]
  (r4) -- node[left=6pt] {\tiny{$m$}} (R);
\end{tikzpicture}

%% file: tree_n.tex
\begin{tikzpicture}[scale=0.25,baseline={([yshift=-.5ex]current bounding box.center)}]
\def\cz{5}
\def\wi{0.5}

\newcommand{\ci}[1]{	
	\fill[black] (#1) circle (\cz pt);
	\draw (#1) circle (\cz pt);
}

\coordinate (R) at (0,0);

\coordinate (r1) at (0,-1);
\coordinate (r2) at (0,-2);
\coordinate (r3) at (0,-4);
\coordinate (r4) at (0,-5);

\draw (R) to (r1);
\draw (r1) to (r2);
\draw[dotted] (r2) to (r3);
\draw (r3) to (r4);

\ci{R}
\ci{r1}
\ci{r2}
\ci{r3}
\ci{r4}

\draw[decoration={brace,raise=5pt},decorate]
  (r4) -- node[left=6pt] {\tiny{$n$}} (R);
\end{tikzpicture}

%% file: tree_nm_6.tex
\begin{tikzpicture}[scale=0.25,baseline={([yshift=-.5ex]current bounding box.center)}]
\def\cz{5}
\def\wi{0.5}

\newcommand{\ci}[1]{	
	\fill[black] (#1) circle (\cz pt);
	\draw (#1) circle (\cz pt);
}

\coordinate (R) at (0,0);

\coordinate (r1) at (0,-1);
\coordinate (r2) at (0,-2);
\coordinate (r3) at (0,-4);
\coordinate (r4) at (0,-5);
\coordinate (r5) at (-1,-6);
\coordinate (r6) at (1,-6);
\coordinate (r7) at (0,-7);
\coordinate (r13) at (0,-8);
\coordinate (r14) at (0,-9);
\coordinate (r15) at (0,-11);
\coordinate (r16) at (0,-12);
\coordinate (r8) at (2,-7);
\coordinate (r9) at (2,-8);
\coordinate (r10) at (2,-9);
\coordinate (r11) at (2,-11);
\coordinate (r12) at (2,-12);
\draw (R) to (r1);
\draw (r1) to (r2);
\draw[dotted] (r2) to (r3);
\draw (r3) to (r4);
\draw (r4) to (r5);
\draw (r4) to (r6);
\draw (r6) to (r7);%
\draw (r6) to (r8);
\draw (r8) to (r9);
\draw (r9) to (r10);
\draw[dotted] (r10) to (r11);
\draw (r11) to (r12);
\draw (r7) to (r13);
\draw (r13) to (r14);
\draw[dotted] (r14) to (r15);
\draw (r15) to (r16);
\ci{R}
\ci{r1}
\ci{r2}
\ci{r3}
\ci{r4}
\ci{r5}
\ci{r6}
\ci{r7}
\ci{r8}
\ci{r9}
\ci{r10}
\ci{r11}
\ci{r12}
\ci{r13}
\ci{r14}
\ci{r15}
\ci{r16}

\draw[decoration={brace,mirror, raise=5pt},decorate]
  (r4) -- node[right=6pt] {\tiny{$m-i+n-j-2$}} (R);

\draw[decoration={brace,mirror, raise=5pt},decorate]
  (r12) -- node[right=6pt] {\tiny{$j$}} (r8);

\draw[decoration={brace,mirror, raise=5pt},decorate]
  (r16) -- node[right=6pt] {\tiny{$i$}} (r7);
  
\end{tikzpicture}

%% file: tree_nm_7.tex
\begin{tikzpicture}[scale=0.25,baseline={([yshift=-.5ex]current bounding box.center)}]
\def\cz{5}
\def\wi{0.5}

\newcommand{\ci}[1]{	
	\fill[black] (#1) circle (\cz pt);
	\draw (#1) circle (\cz pt);
}

\coordinate (R) at (0,0);

\coordinate (r1) at (0,-1);
\coordinate (r2) at (0,-2);
\coordinate (r3) at (0,-4);
\coordinate (r4) at (0,-5);
\coordinate (r5) at (-2,-6);
\coordinate (r6) at (2,-6);
\coordinate (r7) at (0,-6);
\coordinate (r12) at (0,-7);
\coordinate (r13) at (0,-8);
\coordinate (r14) at (0,-10);
\coordinate (r15) at (0,-11);
\coordinate (r8) at (2,-7);
\coordinate (r9) at (2,-8);
\coordinate (r10) at (2,-10);
\coordinate (r11) at (2,-11);
\draw (R) to (r1);
\draw (r1) to (r2);
\draw[dotted] (r2) to (r3);
\draw (r3) to (r4);
\draw (r4) to (r5);
\draw (r4) to (r6);
\draw (r4) to (r7);
\draw (r6) to (r8);
\draw (r8) to (r9);
\draw[dotted]  (r9) to (r10);
\draw(r10) to (r11);
\draw (r7) to (r12);
\draw (r12) to (r13);
\draw[dotted]  (r3) to (r14);
\draw(r14) to (r15);
\ci{R}
\ci{r1}
\ci{r2}
\ci{r3}
\ci{r4}
\ci{r5}
\ci{r6}
\ci{r7}
\ci{r8}
\ci{r9}
\ci{r10}
\ci{r11}
\ci{r12}
\ci{r13}
\ci{r14}
\ci{r15}
\draw[decoration={brace,mirror, raise=5pt},decorate]
  (r4) -- node[right=6pt] {\tiny{$m-i+n-j-1$}} (R);

\draw[decoration={brace,mirror, raise=5pt},decorate]
  (r11) -- node[right=6pt] {\tiny{$j$}} (r6);
 
\draw[decoration={brace,mirror, raise=5pt},decorate]
  (r15) -- node[right=6pt] {\tiny{$i$}} (r7);

\end{tikzpicture}

%% file: tree_5_5.tex
\begin{tikzpicture}[scale=0.25,baseline={([yshift=-.5ex]current bounding box.center)}]
\def\cz{5}
\def\wi{0.5}

\newcommand{\ci}[1]{	
	\fill[black] (#1) circle (\cz pt);
	\draw (#1) circle (\cz pt);
}

\coordinate (R) at (0,0);

\coordinate (r1) at (\wi,-1);
\coordinate (l1) at (-\wi,-1);
\coordinate (r2) at (\wi,-2);
\coordinate (r3) at (\wi,-3);

\draw (R) to (r1);
\draw (R) to (l1);
\draw (r1) to (r2);
\draw (r2) to (r3);

\ci{R}
\ci{r1}
\ci{l1}
\ci{r2}
\ci{r3}
\end{tikzpicture}

%% file: tree_5_6.tex
\begin{tikzpicture}[scale=0.25,baseline={([yshift=-.5ex]current bounding box.center)}]
\def\cz{5}
\def\wi{0.5}

\newcommand{\ci}[1]{	
	\fill[black] (#1) circle (\cz pt);
	\draw (#1) circle (\cz pt);
}

\coordinate (R) at (0,0);

\coordinate (r1) at (\wi,-1);
\coordinate (l1) at (-\wi,-1);
\coordinate (r2) at (\wi+\wi,-2);
\coordinate (l2) at (0,-2);

\draw (R) to (r1);
\draw (R) to (l1);
\draw (r1) to (r2);
\draw (r1) to (l2);

\ci{R}
\ci{r1}
\ci{l1}
\ci{r2}
\ci{l2}
\end{tikzpicture}

%% file: tree_5_3.tex
\begin{tikzpicture}[scale=0.25,baseline={([yshift=-.5ex]current bounding box.center)}]
\def\cz{5}
\def\wi{0.5}

\newcommand{\ci}[1]{	
	\fill[black] (#1) circle (\cz pt);
	\draw (#1) circle (\cz pt);
}

\coordinate (R) at (0,0);

\coordinate (r1) at (0,-1);
\coordinate (r2) at (\wi,-2);
\coordinate (l2) at (-\wi,-2);
\coordinate (r3) at (\wi,-3);

\draw (R) to (r1);
\draw (r1) to (r2);
\draw (r1) to (l2);
\draw (r2) to (r3);

\ci{R}
\ci{r1}
\ci{r2}
\ci{r3}
\ci{l2}
\end{tikzpicture}

%% file: tree_5_8.tex
\begin{tikzpicture}[scale=0.25,baseline={([yshift=-.5ex]current bounding box.center)}]
\def\cz{5}
\def\wi{0.5}

\newcommand{\ci}[1]{	
	\fill[black] (#1) circle (\cz pt);
	\draw (#1) circle (\cz pt);
}

\coordinate (R) at (0,0);

\coordinate (r1) at (\wi+\wi,-1);
\coordinate (c1) at (0,-1);
\coordinate (l1) at (-\wi-\wi,-1);
\coordinate (r2) at (\wi+\wi,-2);

\draw (R) to (r1);
\draw (R) to (c1);
\draw (R) to (l1);
\draw (r1) to (r2);

\ci{R}
\ci{r1}
\ci{c1}
\ci{l1}
\ci{r2}
\end{tikzpicture}

%% file: tree_5_2.tex
\begin{tikzpicture}[scale=0.25,baseline={([yshift=-.5ex]current bounding box.center)}]
\def\cz{5}
\def\wi{0.5}

\newcommand{\ci}[1]{	
	\fill[black] (#1) circle (\cz pt);
	\draw (#1) circle (\cz pt);
}

\coordinate (R) at (0,0);

\coordinate (r1) at (0,-1);
\coordinate (r2) at (0,-2);
\coordinate (r3) at (\wi,-3);
\coordinate (l3) at (-\wi,-3);

\draw (R) to (r1);
\draw (r1) to (r2);
\draw (r2) to (r3);
\draw (r2) to (l3);

\ci{R}
\ci{r1}
\ci{r2}
\ci{r3}
\ci{l3}
\end{tikzpicture}